\documentclass[11pt, reqno]{amsart}

\usepackage{amsmath,amssymb,amsthm,amsfonts,verbatim}
\usepackage{microtype}
\usepackage[all,2cell]{xy}
\usepackage{mathtools}
\usepackage{graphicx}
\usepackage{pinlabel}
\usepackage{hyperref}
\usepackage{mathrsfs}
\usepackage{color}
\usepackage[dvipsnames]{xcolor}
\usepackage{enumerate}
\usepackage{cite}
\usepackage{soul}

\CompileMatrices

\usepackage[top=1.2in,bottom=1.2in,left=1in,right=1in]{geometry}

\usepackage{hyperref} 
\hypersetup{          
	colorlinks=true, breaklinks, linkcolor=[RGB]{51 102 204}, filecolor=Orchid, urlcolor=[RGB]{51 102 204},
	citecolor=Orchid, linktoc=all, }
\usepackage[nameinlink]{cleveref}

\theoremstyle{plain}
\newtheorem{theorem}{Theorem}[section]
\newtheorem{maintheorem}{Theorem}

\newtheorem{maincor}[maintheorem]{Corollary}

\newtheorem{proposition}[theorem]{Proposition}
\newtheorem{lemma}[theorem]{Lemma}

\newtheorem{corollary}[theorem]{Corollary}

\theoremstyle{definition}
\newtheorem{definition}[theorem]{Definition}
\newtheorem{example}[theorem]{Example}

\newtheorem{remark}[theorem]{Remark}

\newcommand{\nc}{\newcommand}
\nc{\dmo}{\DeclareMathOperator}

\nc{\Q}{\mathbb{Q}}
\nc{\F}{\mathbb{F}}
\nc{\R}{\mathbb{R}}
\nc{\Z}{\mathbb{Z}}
\nc{\C}{\mathbb{C}}
\nc{\N}{\mathbb{N}}
\nc{\Ell}{\mathcal{L}}
\nc{\M}{\mathcal{M}}
\nc{\K}{\mathcal{K}}
\nc{\I}{\mathcal{I}}
\nc{\T}{\mathcal T}
\nc{\U}{\mathcal U}
\nc{\disk}{\mathbb{D}}
\nc{\hyp}{\mathbb{H}}

\nc{\CP}{\mathbb{CP}}
\nc{\RP}{\mathbb{RP}}
\dmo{\Mod}{Mod}
\dmo{\PMod}{PMod}
\dmo{\LMod}{LMod}
\dmo{\Diff}{Diff}
\dmo{\Homeo}{Homeo}
\dmo{\dist}{dist}
\dmo\BDiff{BDiff}
\dmo\SO{SO}
\dmo\Hom{Hom}
\dmo\SL{SL}
\dmo\rank{rank}
\dmo\sig{sig}
\dmo\Out{Out}
\dmo\Aut{Aut}
\dmo\Inn{Inn}
\dmo\GL{GL}
\dmo\PGL{PGL}
\dmo\Gr{Gr}
\dmo\PSL{PSL}
\dmo\BHomeo{BHomeo}
\dmo\EHomeo{EHomeo}
\dmo\EDiff{EDiff}
\dmo\Disc{Disc}
\dmo\Aff{Aff}

\renewcommand{\bar}{\overline}

\dmo\Teich{Teich}
\dmo\Fix{Fix}
\nc{\pair}[1]{\ensuremath{\left\langle #1 \right\rangle}}
\nc{\abs}[1]{\ensuremath{\left| #1 \right|}}
\nc{\action}{\circlearrowright}
\nc{\norm}[1]{\left | \left | #1 \right | \right |}
\nc{\abcd}[4]{\ensuremath{\left(\begin{array}{cc} #1 & #2 \\ #3 & #4 \end{array}\right)}}
\nc{\into}\hookrightarrow
\dmo{\Isom}{Isom}
\nc{\normal}{\vartriangleleft}
\dmo{\Vol}{Vol}
\dmo{\im}{Im}
\dmo{\Push}{Push}
\dmo{\Conf}{Conf}
\dmo{\PConf}{PConf}
\dmo{\PB}{PB}
\dmo{\id}{id}
\dmo{\Jac}{Jac}
\dmo{\Pic}{Pic}
\dmo{\Stab}{Stab}
\dmo{\Arf}{Arf}
\dmo{\End}{End}
\dmo{\Gal}{Gal}
\dmo{\lcm}{lcm}
\dmo{\ab}{ab}
\dmo{\opp}{op}
\dmo{\SU}{SU}
\dmo{\OT}{\Omega \mathcal{T}}
\dmo{\OM}{\Omega \mathcal{M}}
\dmo{\PH}{\mathbb{P}\mathcal{H}}
\dmo{\spin}{spin}
\dmo{\even}{even}
\dmo{\odd}{odd}
\dmo{\comp}{\mathcal{H}}
\dmo{\Mgk}{\mathcal{M}_{g, \underline{\kappa}}}
\dmo{\orb}{orb}
\dmo{\AJ}{AJ}
\dmo{\Ck}{\mathsf{C}(\underline{\kappa})}
\dmo{\Int}{Int}
\dmo{\pr}{pr}
\dmo{\lab}{lab}
\dmo{\Sym}{Sym}
\dmo{\Ann}{Ann}
\dmo{\Rad}{Rad}
\dmo{\Ind}{Ind}
\dmo{\Div}{Div}
\dmo{\Res}{Res}
\dmo{\Hur}{Hur}
\dmo{\vcd}{vcd}

\nc{\Span}[1]{\operatorname{Span}(#1)}
\newcommand{\onto}{\twoheadrightarrow}

\renewcommand{\epsilon}{\varepsilon}
\renewcommand{\tilde}{\widetilde}
\renewcommand{\le}{\leqslant}
\nc{\coloneq}{\mathrel{\mathop:}\mkern-1.2mu=}
\nc{\margin}[1]{\marginpar{\scriptsize #1}}
\nc{\para}[1]{\medskip\noindent\textbf{#1.}}
\definecolor{myblue}{RGB}{102,153, 255}
\definecolor{myred}{RGB}{204,0,0}
\definecolor{mygreen}{RGB}{0,204,0}
\definecolor{myorange}{RGB}{255,102,0}
\definecolor{mypurple}{RGB}{138,43,226}
\nc{\red}[1]{\textcolor{myred}{#1}}
\nc{\blue}[1]{\textcolor{myblue}{#1}}

\newcommand{\boldit}[1]{\textit{\textbf{#1}}}

\title[Connected components of the topological surgery graph]{Connected components of the topological surgery graph of a unicellular collection}
\author{Nick Salter}
\author{Abdoul Karim Sane}
\date{August 11, 2023}

\begin{document}
\maketitle
\begin{abstract}
    A unicellular collection on a surface is a collection of curves whose complement is a single disk. There is a natural surgery operation on unicellular collections, endowing the set of such with a graph structure where the edge relation is given by surgery. Here we determine the connected components of this graph, showing that they are enumerated by a certain homological ``surgery invariant''. Our approach is group-theoretic and proceeds by understanding the action of the mapping class group on unicellular collections. In the course of our arguments, we determine simple generating sets for the stabilizer in the mapping class group of a mod-$2$ homology class, which may be of independent interest. 
\end{abstract}

\section{Introduction}
Let $\Sigma_g$ be the closed orientable surface of genus $g$. A \boldit{unicellular collection} on $\Sigma_g$ is the isotopy class of a collection $\Gamma$ of (not necessarily simple) closed curves in minimal position such that $\Sigma_g\setminus\Gamma$ is a single disk. Unicellular collections feature in Harer--Zagier's computation of the Euler characteristic of the moduli space of curves \cite{har}, and more generally form a point of interaction between surface topology and combinatorics.

In \cite{karim}, the second-named author introduced an operation on unicellular collections called {\em surgery} (see \Cref{bground}), that transforms a unicellular collection into a new one. This leads to a graph $\widetilde{\mathcal K}_{d_4,g}$ called the  \textit{topological surgery graph} whose vertices are unicellular collections on $\Sigma_g$ and whose edges are given by surgeries; the subscript $d_4$ records the fact that each vertex of a unicellular collection is $4$-valent when considered as a graph embedded in $\Sigma_g$. The mapping class group $\Mod(\Sigma_g)$ acts on $\widetilde{\mathcal K}_{d_4,g}$ and the quotient graph, denoted $\mathcal K_{d_4,g}$, is called the \textit{combinatorial surgery graph}; its vertices are called \textit{unicellular maps} and by definition are $\Mod(\Sigma_g)$-orbits of unicellular collections.

Given a unicellular collection $\Gamma$, its mod-$2$ homology class $[\gamma]$ turns out to be a surgery invariant. Consequently, the graph $\widetilde{\mathcal K}_{d_4,g}$ is disconnected for all $g\ge1$. For a fixed unicellular collection $\Gamma$, we denote by $\Mod(\Sigma_g)[\Gamma]$ the \boldit{stabilizer subgroup} of mapping classes that fix the connected component of $\Gamma$. Let $\Mod(\Sigma_g)[\gamma]$ denote the stabilizer of the surgery invariant $[\gamma]\in H_1(\Sigma_g,\mathbb{Z}/2\mathbb{Z})$ of $\Gamma$. Our main theorem is:   

\begin{maintheorem}\label{maintheorem}
For every $g \ge 3$, the stabilizer subgroup $\Mod(\Sigma_g)[\Gamma]$ associated to the connected component of $\widetilde{\mathcal K}_{d_4,g}$ containing $\Gamma$ is the group $\Mod(\Sigma_g)[\gamma]$, where $[\gamma] \in H_1(\Sigma_g; \Z/2\Z)$ is the associated surgery invariant. Consequently, unicellular collections $\Gamma$ and $\Gamma'$ are related by a sequence of surgeries if and only if there is an equality $[\gamma] = [\gamma'] \in H_1(\Sigma_g;\Z/2\Z)$ of surgery invariants.
\end{maintheorem}

\begin{maincor} For every $g\ge1$, the topological surgery graph $\widetilde{K}_{d_4,g}$ has $2^{2g}-1$ connected components, indexed by the set $H_1(\Sigma_g;\Z/2\Z)\setminus \{0\}$ of surgery invariants. 
\end{maincor}

Unicellular collections are instances of a more general notion of \boldit{unicellular graphs}, which are by definition (no longer necessarily $4$-valent) graphs embedded in $\Sigma_g$ with complement a single disk. In this broader setting, the surgery operation is still sensible, and the surgery invariant still exists for any unicellular graph all of whose vertices have even degree, but it is not known whether the connected components are still enumerated by this single mod-$2$ invariant.

The proof of \Cref{maintheorem} goes through the exhibition of an explicit generating set of the stabilizer subgroup of a mod-$2$ homology class $[\gamma]$. Generating sets for these subgroups of $\Mod(\Sigma_g)$ were determined in \cite{Dey} as part of a broader program to understand {\em liftable mapping class groups}, i.e. the set of mapping classes on $\Sigma_g$ that {\em lift} along some covering map $\Sigma_h \to \Sigma_g$. The generating sets we find here are simpler than those of \cite{Dey} (ours consist of fewer generators and only Dehn twists and their squares) and in some sense are the simplest possible. We imagine that this result may be of independent interest; we record it below. For a simple closed curve $\xi \subset\Sigma_g$, the subgroup $\Mod(\Sigma_g,\xi)$ is the group of mapping classes preserving $\xi$ as an {\em oriented} isotopy class of curve; it is a quotient of the mapping class group of the surface $\Sigma_{g-1,2}$ obtained by cutting along $\xi$ and hence is generated by the image of any generating set for $\Mod(\Sigma_{g-1,2})$, e.g. the Humphries generators (cf. \cite[Figure 4.10]{farmar}).

\para{\Cref{theorem:genset}} \textit{Let $g \ge 3$ be given, and let $x \in H_1(\Sigma_g; \Z/2\Z)$ be nonzero. Then $\Mod(\Sigma_g)[x]$ is generated by $T_\eta^2$ and $\Mod(\Sigma_g,\xi)$, where $\xi$ is any simple closed curve satisfying $[\xi] = x \in H_1(\Sigma_g; \Z/2\Z)$, and $\eta$ is any simple closed curve with geometric intersection $i(\xi, \eta) = 1$.}\\

To prove \Cref{maintheorem}, we combine \Cref{theorem:genset} with a set of techniques for expressing certain simple surgeries as Dehn twists. We show that there are enough of these to express each of the generators for $\Mod(\Sigma_g)[\Gamma]$ required by \Cref{theorem:genset} as surgeries.\\

Unicellular graphs/maps appear in various places throughout geometry and combinatorics. For instance, when $\Sigma_g$ is endowed with a hyperbolic metric, unicellular graphs appear as the cut locus of the exponential map. This fact is used to build a cell decomposition of the Teichmüller space $\mathcal{T}_{g,1}$ and ultimately the moduli space $\mathcal{M}_g^1$ of Riemann surfaces of genus $g$ equipped with one marked point $p$. In a more combinatorial direction, unicellular maps have been studied by Walsh--Lehman to provide the counting formulas of \cite{leh, leh2}. The cell decomposition of $\mathcal{M}_g^1$ and counting formulas for unicellular maps led J. Harer and D. Zagier \cite{har} to compute the Euler characteristic of the moduli space $\mathcal{M}_g^1$.   

$\Mod(\Sigma_g)$ acts on a wide variety of discrete sets, e.g. the set of isotopy classes of curves, $H_1(\Sigma_g;\Z)$, the set of ``$r$-spin structures'' studied in \cite{nicktoric,strata2}), and the set of unicellular collections as studied in this paper. The results of this article fit into the general theme of studying the stabilizer subgroups of such actions. A virtue of the results obtained in \Cref{theorem:genset} is that they are {\em coordinate-free} in the sense that the generators we obtain for $\Mod(\Sigma_g)[\Gamma] = \Mod(\Sigma_g)[x]$ do not make reference to any particular configuration of curves or depiction of the surface, and are thus can be exhibited in specific examples without lengthy picture calculations. Coordinate-free generating sets for other types of stabilizer subgroups have been obtained over the past few years for the mapping class group itself \cite{nicklei} as well as for $r$-spin mapping class groups and the closely-related {\em framed} mapping class groups \cite{strata3}.

\para{Outline} In \Cref{bground}, we recall some basic facts about unicellular graphs and the surgery operation. In \Cref{section:stabilizer}, we define the stabilizer subgroup associated to a connected component of $\tilde{\mathcal K}_{d_4,g}$ and exhibit certain Dehn twists as the result of performing surgeries. Finally in \Cref{section:genset}, we prove \Cref{theorem:genset} and use this and the results of \Cref{section:stabilizer} to prove \Cref{maintheorem}.

\para{Acknowledgements} This project was begun at the 2022 AMS Southeast Spring Sectional meeting in the special session on geometric group theory. The authors would like to thank the organizers of this session, in particular Dan Margalit, for facilitating an introduction and for some valuable comments on a draft. The first-named author is supported in part by NSF grant no. DMS-2153879. The second-named author thanks the Department of Mathematics at Georgia Tech and the Department of Mathematics at the University of Dakar for support and hospitality during the period this paper was prepared.

\section{Background on unicellular graphs and unicellular maps}\label{bground}
We recall that $\Sigma_g$ denotes a closed oriented surface of genus $g$ and $\Mod(\Sigma_g)$ is the mapping class group of $\Sigma_g$. A \boldit{unicellular graph} on $\Sigma_g$ is (the isotopy class of) a graph $G=(V,E)$ embedded on $\Sigma_g$ whose complement is a single disk. The \boldit{degree partition} of a unicellular graph is the finite set $d:=\{d_1,\dots,d_n\}$ of positive integers where $d_i$ records the valence of the corresponding vertex $v_i$. Using the Euler characteristic of $\Sigma_g$, we obtain the following relation:
\[
2|E|=\displaystyle{\sum_id_i=2|V| + 4g-2}.
\]
When $G$ is a regular 4-valent unicellular graph we speak of a \boldit{unicellular collection}, since it can be seen as a collection of closed curves with only double points; in this case we denote a unicellular collection by $\Gamma$. For unicellular collections, the relation above reduces to the conditions 
\[
|V|=2g-1 \qquad \mbox{and} \qquad E=4g-2.  
\]
 A \boldit{unicellular map} is the $\Mod(\Sigma_g)$-orbit of a unicellular graph; that is, we consider unicellular graphs up to orientation preserving homeomorphisms. Unicellular maps are combinatorial objects and can be described in many ways. One of them is to associate a \boldit{coding} to a unicellular map. Since $\Sigma_g-G$ is an oriented\footnote{the orientation inherited by the one on $\Sigma_g$} polygon, the sides of $\Sigma_g-G$ can be labelled following the orientation in such way that the two sides coming from the same edge of $G$ have the same letter with a bar  on the second one (for example $x$ and $\bar{x}$). Doing so, we get a word $W_G$ that we refer to as the {\em coding} of the unicellular map $G$. Two codings are considered equivalent if one can be obtained from the other by relabeling and cyclic permutation of the letters. In the coding of a unicellular map $G$, a letter $x$ refer to an oriented edge of $G$ and we use the notation $(x,\bar{x})$ to denote the corresponding non oriented edge. 

 \begin{example} The following is a unicellular collection $\Gamma$ on a genus $2$ surface. On the right, we have the complement $\Sigma_2-\Gamma$ where the identified sides have the same color. The coding corresponding to $\Gamma$ is $W_{\Gamma}=a\hspace{0.1cm}b\hspace{0.1cm}c\hspace{0.1cm}d\hspace{0.1cm}\bar{b}\hspace{0.1cm}e\hspace{0.1cm}\bar{c}\hspace{0.1cm}f\hspace{0.1cm}\bar{a}\hspace{0.1cm}\bar{e}\hspace{0.1cm}\bar{f}\hspace{0.1cm}\bar{d}$.  
 \begin{figure}[htbp]
\begin{center}
\includegraphics[scale=0.2]{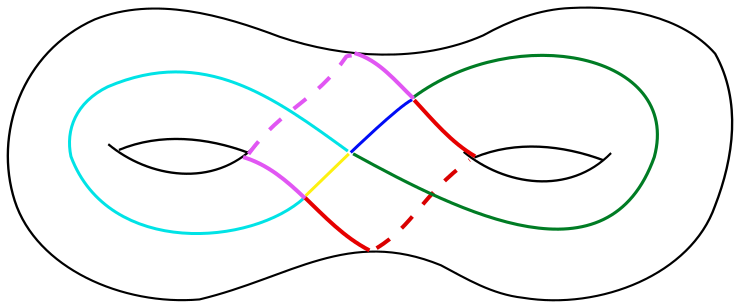} \hspace{2cm}
\includegraphics[scale=0.2]{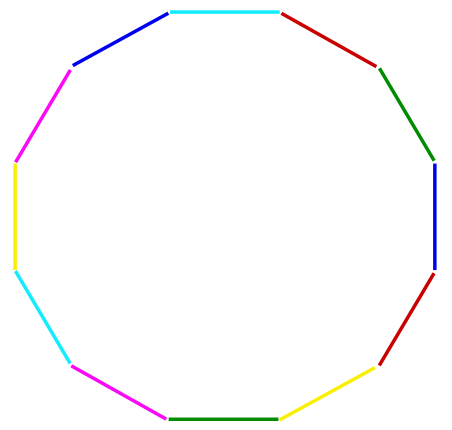}\put(-46,87){\tiny{$a$}}\put(-25,81){\tiny{$b$}}\put(-7,63){\tiny{$c$}}\put(-2,41){\tiny{$d$}}\put(-6,19){\tiny{$\bar{b}$}}\put(-25,3){\tiny{$e$}}\put(-46,-5){\tiny{$\bar{c}$}}\put(-72,2){\tiny{$f$}}\put(-90,19){\tiny{$\bar{a}$}}\put(-94,41){\tiny{$\bar{e}$}}\put(-88,63){\tiny{$\bar{f}$}}\put(-72,81){\tiny{$\bar{d}$}}
\label{default}
\caption{A unicellular map and its complement.}
\end{center}
\end{figure}
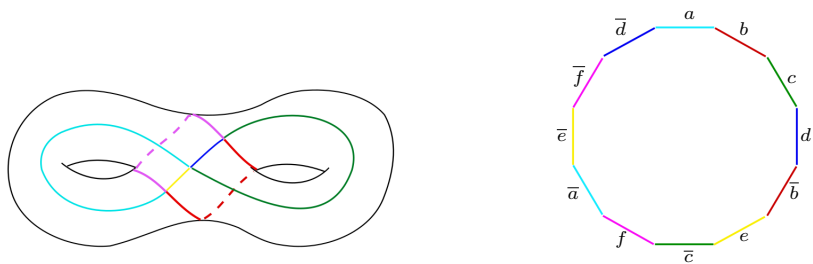
 \end{example}

The following result asserts that the coding is a complete invariant of the unicellular map. This is ultimately a consequence of the classification of surfaces; see \cite[Proposition 2.1]{karim} for a proof.
\begin{proposition}
Two unicellular graphs on $\Sigma_g$ are in the same  $\Mod(\Sigma_g)$ orbit if and only they have the same coding (up to relabeling and cyclic permutation). 
\end{proposition}

 In \cite{karim}, the second author introduces an operation called \boldit{surgery} on unicellular graphs. Given a unicellular graph $G$ on $\Sigma_g$ and two oriented edges $x$ and $y$ of $G$, there exists a unique  arc $\lambda_{x,y}$ (up to relative homotopy with endpoints gliding in $x$ and $y$) from the right-hand side of $x$ to the right-hand side of $y$. We get a new graph $\sigma_{x,y}(G)$ by cutting and gluing $x$ and $y$ along $\lambda_{x,y}$; note that $\sigma_{x,y}(G)$ may no longer be unicellular.
 
 \begin{figure}[htbp]
\begin{center}
\includegraphics[scale=0.45]{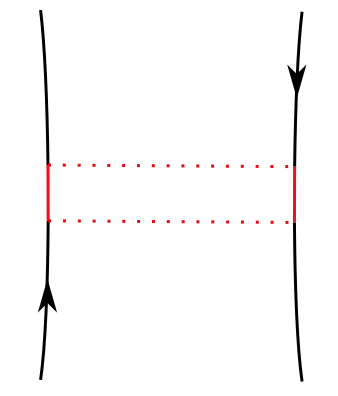}\hspace{3cm}
\includegraphics[scale=0.45]{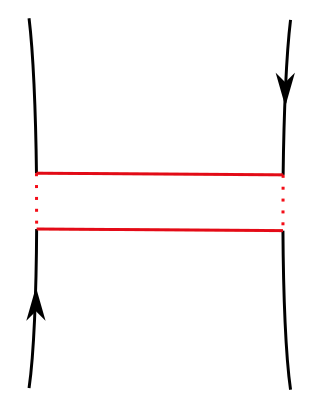}
\put(-125,40){\Large{$\longrightarrow$}}
\caption{A 0-surgery obtained by cutting out $S^0\times D^1$ (the two vertical components) from $x$ and $y$, and gluing in $\mathbb{D}^1\times \mathbb{S}^0$ (the two horizontal components). The two vertical components of $S^0\times D^1$ may belong to the same (unoriented) edge $(x, \bar{x})$; in this case, the endpoints of $\lambda_{x,\bar{x}}$ are required to be disjoint.} 
\label{0surg}
\end{center}
\end{figure}

 When $\sigma_{x,y}(G)$ is still a unicellular graph, we call the operation a \boldit{surgery} on $G$ along $x$ and $y$.
 
 In \cite{karim} we gave a necessary and sufficient condition on oriented edges along which surgeries can be applied.  
 \begin{definition} Let $G$ be a unicellular graph on $\Sigma_g$ and $W_G$ the coding of the unicellular map defined by $G$. Let $x$ and $y$ be two oriented edges. We say that $x$ and $y$ are \boldit{intertwined} if they appear in the coding as follows:
 \[
 W_G=w_1\boldsymbol{x}w_2\boldsymbol{\bar{x}}w_3\boldsymbol{y}w_4\boldsymbol{\bar{y}}.
 \] 
 This is equivalent to the topological condition of the arcs $\lambda_{x,y}$ and $\lambda_{\bar{x},\bar{y}}$ intersecting once on $\Sigma_g$. 
 \end{definition}  
 
 The following gives a complete characterization of when surgery can be performed: 
 \begin{lemma}[A. K. Sane \cite{karim}]
 Let $G$ be a unicellular graph on $\Sigma_g$ and $x$ and $y$ be two oriented edges. Then, a surgery operation $\sigma_{x,y}(G)$ along $x$ and $y$ exists if and only if $x$ and $y$ are intertwined. Moreover, if $w_1\boldsymbol{x}w_2\boldsymbol{\bar{x}}w_3\boldsymbol{y}w_4\boldsymbol{\bar{y}}$ is the coding associated to $G$, then $w_3\boldsymbol{x}w_2\boldsymbol{\bar{x}}w_1\boldsymbol{y}w_4\boldsymbol{\bar{y}}$ is the one associated to $\sigma_{x,y}(G)$. 
 \end{lemma}
 \begin{figure}[htbp]
\begin{center}
\includegraphics[scale=0.25]{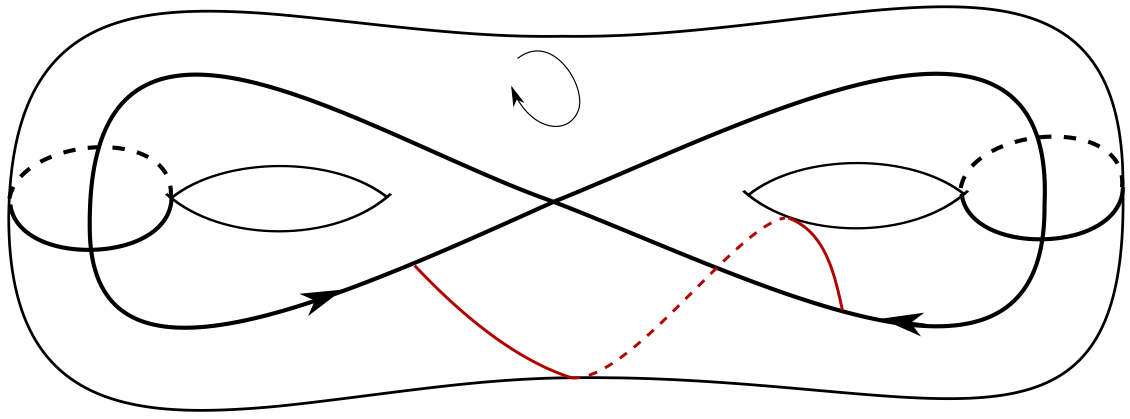}\hspace{1cm}
\includegraphics[scale=0.25]{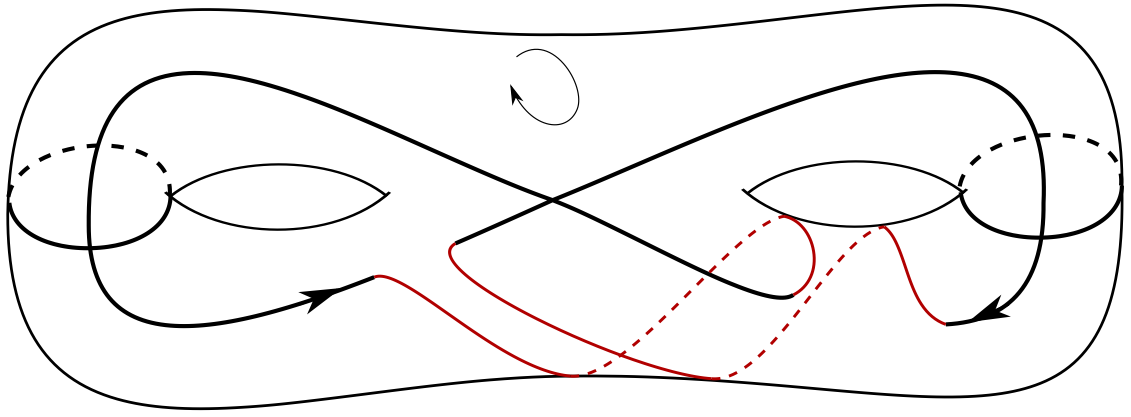}\put(-40,30){$\rightarrow$}
\caption{Example of a surgery on a unicellular collection}
\label{default2}
\end{center}
\end{figure}

  \begin{remark} Let $G$ be a unicellular graph and $x$, $y$ two intertwined edges. 
 \begin{itemize}
 \item Note that surgery is performed on {\em oriented} edges. In other words, $x$ and $y$ can be intertwined but not $x$ and $\bar{y}$.
 \item It is apparent that $x$ and $y$ are intertwined if and only if $\bar{x}$ and $\bar{y}$ are. 
 \item The situation where $y=\bar{x}$ is allowed, subject to the following convention. The arc $\lambda_{x,\bar{x}}$ is chosen so that the endpoint $\lambda_{x,\bar{x}}(0)$ precedes $\lambda_{x,\bar{x}}(1)$ when running along the oriented edge $x$. We will see below in \Cref{lemma:twists}.1 that in this case, $\sigma_{x,\bar{x}}(G)$ is always a unicellular graph and is in the same $\Mod(\Sigma_g)$-orbit as $G$. 
 \end{itemize}
 \end{remark}
 
 Now, we define two natural graphs associated to surgery on unicellular graphs/maps. Let $\widetilde{\mathcal{U}}_{d,g}$ (respectively $\mathcal{U}_{d,g}$) be the set of all unicellular graphs (respectively, unicellular maps) on $\Sigma_g$ with degree partition $d$. The \boldit{topological surgery graph} $\widetilde{\mathcal{K}}_{d,g}$ is the graph whose vertices are elements of $\widetilde{\mathcal{U}}_{d,g}$ with two vertices sharing an edge if there is a surgery taking one to the other. The mapping class group $\Mod(\Sigma_g)$ acts on $\widetilde{\mathcal{K}}_{d,g}$ by simplicial automorphisms and the quotient graph, denoted by $\mathcal{K}_{d,g}$, is called the \boldit{combinatorial surgery graph}. The graph $\mathcal{K}_{d,g}$ has finitely many vertices corresponding to elements in $\mathcal{U}_{d,g}$; edges are again given by surgeries. In \cite{karim,karim2}, the author prove several results on those graphs:

 \begin{theorem}[A. K. Sane \cite{karim,karim2}]\label{theorem:previous} Let $d_3:=(3,\dots,3)$ and $d_4:=(4,\dots,4)$. 
 \begin{enumerate}
 \item The graphs $\mathcal{K}_{d_3,g}$ and $\mathcal{K}_{d_4,g}$ are connected for every $g\geq 1$. Moreover, their diameters are at most quadratic functions of the genus.
 \item For every $g\geq 2$, the graph $\widetilde{\mathcal{K}}_{d_3,g}$ is connected. 
 \item For every $g\geq 2$, the graph $\widetilde{\mathcal{K}}_{d_4,g}$ is disconnected. 
 \end{enumerate}
 \end{theorem} 
 The fact that $\widetilde{\mathcal{K}}_{d_4,g}$ is disconnected relies on the construction of an invariant of surgery. We will see in the next section how the connectedness problem is related to a certain subgroup of the mapping class group. 
 
\section{Unicellular collections and their stabilizer group in $\Mod(\Sigma_g)$}\label{section:stabilizer}

Let $G$ be a unicellular graph on $\Sigma$ with degree partition $d$. 

\begin{definition} The \boldit{stabilizer group} group of $G$, denoted $\Mod(\Sigma_g)[G]$,  is the group of all mapping class elements that fix the connected component of $\widetilde{\mathcal{K}}_{d,g}$ containing $G$. That is, $\phi\in\Mod(\Sigma_g)[G]$ if $G$ and $\phi(G)$ are related by a sequence of surgeries. 
\end{definition} 
Stabilizer groups are related to the connectedness of topological surgery graphs. In fact, via the {\em Putman trick} \cite{putmantrick}, $\widetilde{\mathcal{K}}_{d,g}$ is connected if and only if the combinatorial graph $\mathcal{K}_{d,g}$ is connected and the stabilizer group $\Mod(\Sigma_g)[G]$ of any $G \in \widetilde{\mathcal{K}}_{d,g}$ is the whole mapping class group. The connectedness of  $\widetilde{\mathcal{K}}_{d_3,g}$ relies on those two facts. But in general, the stabilizer group of a given unicellular graph is a proper subgroup of the mapping class group as we will see.

Suppose all of the vertices of $G$ are of even valence. Then $G$ defines a {\em cycle} in the mod-$2$ singular homology $C_1(\Sigma_g; \Z/2\Z)$, and hence a class $[G] \in H_1(\Sigma_g; \Z/2\Z)$.  

\begin{lemma}[Surgery invariant] Let $G_1$ and $G_2$ be two unicellular graphs, all of whose vertices have even degree. Suppose that $G_1$ and $G_2$ are related by a sequence of surgeries. Then $[G_1] = [G_2]$ in $H_1(\Sigma_g; \Z/2\Z)$.
\end{lemma}
\begin{proof}
Let $\lambda_{x,y}$ be the arc associated to a surgery on $G_1$, and let $\Lambda_{x,y} \cong I \times I$ be the rectangular strip with core $\lambda_{x,y}$ used to perform the surgery. Then there is an equality of mod-$2$ singular chains
\[
\sigma_{x,y}(G_1) = G_1 + \partial \Lambda_{x,y},
\]
showing that $[\sigma_{x,y}(G_1)] = [G_1]$ in $H_1(\Sigma_g; \Z/2\Z)$. The result follows.
\end{proof}

Let $G$ be a unicellular graph with vertices of even degree, and let $\beta$ be a simple closed curve. Here and throughout, denote the Dehn twist about $\beta$ by $T_\beta$. If the mod-$2$ algebraic intersection $\pair{[\beta],[G]}_2$ is nonzero, then the transvection formula for Dehn twists shows that 
\[
[T_\beta(G)] = [G] + [\beta] \ne [G],
\]
showing that $G$ and $T_\beta(G)$ lie in distinct components of the topological surgery graph $\tilde{\mathcal K}_{d,g}$. Thus we find a necessary condition for a Dehn twist $T_\beta$ to be in $\Mod(\Sigma_g)[G]$: necessarily $\pair{[G], [\beta]}_2 = 0$.

\begin{figure}[htbp]
\begin{center}
\includegraphics[scale=0.22]{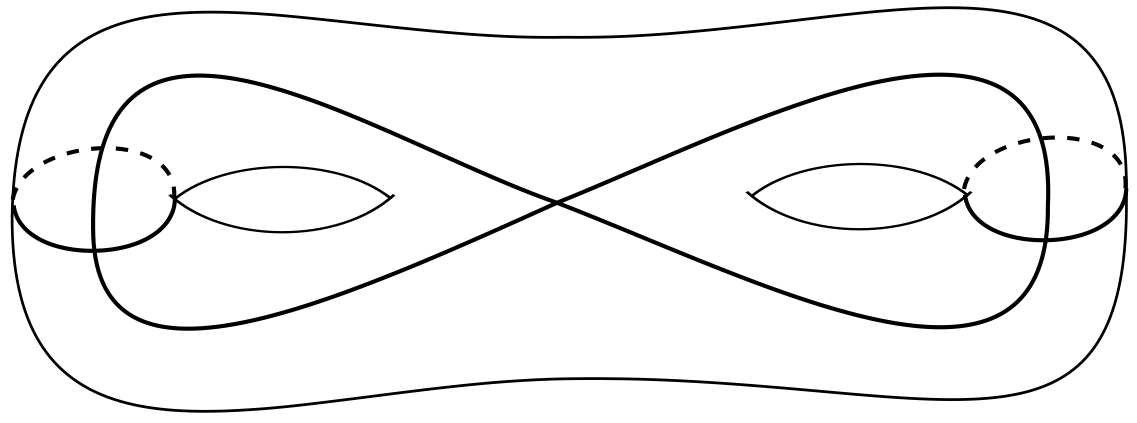}\hspace{1cm}
\includegraphics[scale=0.22]{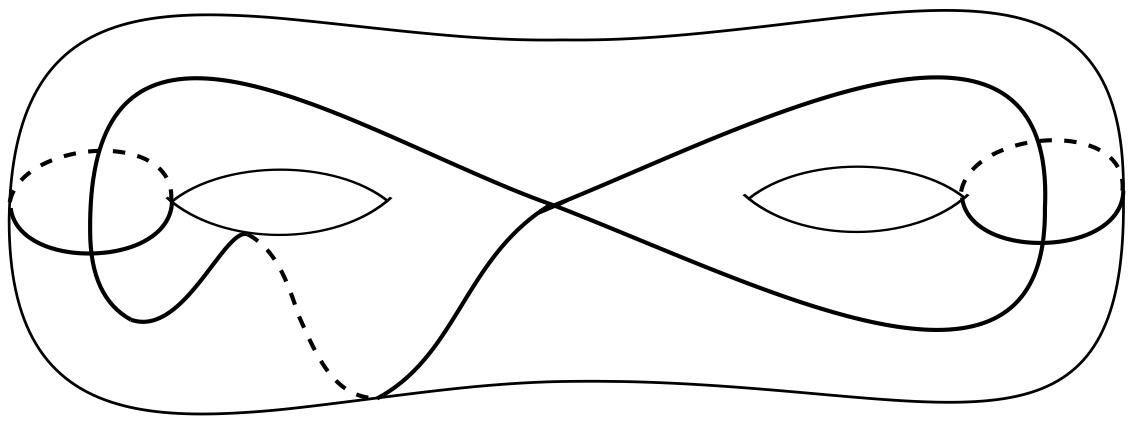}
\caption{Two unicellular collections in different connected components of $\tilde{K}_{(4,4,4),2}$.}
\label{twist}
\end{center}
\end{figure}
A useful criterion to check whether or not two mod-$2$ homology classes $[G_1]$ and $[G_2]$ are equal is to see how they intersect simple curves in $\Sigma_g$. In fact, $[G_1]=[G_2]\mod2$ if and only $\pair{G_1, \alpha}_2=\pair{G_2,\alpha}_2$ for every simple closed curve $\alpha$.

\begin{corollary} \label{corollary:stabsubgroup}
Let $G$ be a unicellular graph with even degree vertices. Then $\Mod(\Sigma_g)[G]$ is a subgroup of $\Mod(\Sigma_g)[x]$, where $x = [G] \in H_1(\Sigma_g; \Z/2\Z)$ and $\Mod(\Sigma_g)[x]$ denotes the stabilizer of $x$ under the action of $\Mod(\Sigma_g)$ on $H_1(\Sigma_g; \Z/2\Z)$. In particular, $\Mod(\Sigma_g)[G]$ is a {\em proper} subgroup of $\Mod(\Sigma_g)$. 
\end{corollary}

The stabilizer group $\Mod(\Sigma_g)[G]$ is related to both the topological surgery graph and the combinatorial surgery graph in the following sense:

\begin{lemma}\label{lemma:conjequal}
Let $G_1$ and $G_2$ be two vertices of $\widetilde{\mathcal{K}}_{d,g}$. If $G_1$ and $G_2$ are in the same component, then the subgroups $\Mod(\Sigma_g)[G_1]$ and $\Mod(\Sigma_g)[G_2]$ of $\Mod(\Sigma_g)$ are equal. If $G_1$ and $G_2$ are not necessarily in the same component but the combinatorial graph $\mathcal{K}_{d,g}$ is connected, then $\Mod(\Sigma_g)[G_1]$ and $\Mod(\Sigma_g)[G_2]$ are conjugate in $\Mod(\Sigma_g)$. 
\end{lemma}
\begin{proof}
Let $\phi\in\Mod(\Sigma_g)[G_1]$; by definition $\phi(G_1)$ and $G_1$ are in the same component. Since $G_1$ and $G_2$ are related by a sequence of surgeries, so are $\phi(G_1)$ and $\phi(G_2)$. It follows that $\phi(G_2)$ and $G_1$ are in the same component, but this component also contains $G_2$. Hence, $\phi\in\Mod(\Sigma_g)[G_2]$ which implies that $\Mod(\Sigma_g)[G_1] \le\Mod(\Sigma_g)[G_2]$. The other inclusion follows the same idea. 

Assume now that $\mathcal{K}_{d,g}$ is connected, but $G_1$ and $G_2$ are no longer necessarily in the same component of $\widetilde{\mathcal{K}}_{d,g}$. Then there exists a sequence of surgeries between $G_1$ and $\phi(G_2)$, for some $\phi\in\Mod(\Sigma_g)$. By the above, $\Mod(\Sigma_g)[G_1]=\Mod(\Sigma_g)[\phi(G_2)]$, and the result now follows from the conjugacy equation $\Mod(\Sigma_g)[\phi(G_2)]=\phi \Mod(\Sigma_g)[G_2]\phi^{-1}$.
\end{proof}
Our main theorem in this article gives a description of $\Mod(\Sigma_g)[G]$ when $G$ is a regular 4-valent graph. We end this section by providing some elements in $\Mod(\Sigma_g)[G]$, namely Dehn twists. Below, when we speak of {\em intersections} of isotopy classes of curves and/or graphs, we mean that there exists a pair of representatives with the specified intersection pattern for which all intersections are transverse.

\begin{lemma}\label{lemma:twists}
Let $G$ be a unicellular graph with vertices of even degree. 
\begin{enumerate}
\item If $\beta$ is a simple curve that intersects $G$ at exactly one point, then $T_{\beta}^2\in\Mod(\Sigma_g)[G]$. 
\item Let $\beta_1$ be a simple curve that intersects $G$ at exactly two points lying on intertwined edges $x \ne y$, and let $\beta_2$ and $\beta_3$ be two simple curves that intersect $G$ exactly once at $x$ and $y$, respectively.  Then $T_{\beta_2}^2 T_{\beta_1}T_{\beta_3}^2\in\Mod(\Sigma_g)[G]$. 
\end{enumerate}
\end{lemma}
\begin{proof}
For (1), assume that $\beta$ intersects $G$ along the non-oriented edge $\{x, \bar{x}\}$. Then \Cref{degn} shows that $T_{\beta}^2(G)$ is given by $\sigma_{x,\bar{x}}(G)$.
\begin{figure}[htbp]
\labellist
\small
\pinlabel $x$ at 10 35
\pinlabel $\bar{x}$ at 20 50
\pinlabel $\beta$ at 53 20
\endlabellist
\begin{center}
\includegraphics{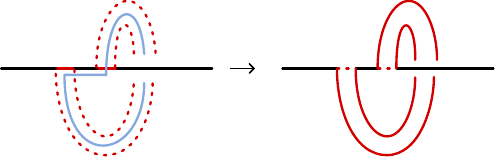}
\caption{Surgery on $G$ between $x$ and $\bar{x}$ is given by the square-twist about the indicated curve $\beta$.}
\label{degn}
\end{center}
\end{figure}

For (2), suppose that $x \neq y$ are intertwined. The sequence \[G\rightarrow \sigma_{x,y}(G)\rightarrow G':=\sigma_{\bar{x},\bar{y}}(G)\] of surgeries is such that $G$ and $G'$ are in the same $\Mod(\Sigma_g)$ orbit. In fact, if $w_1xw_2\bar{x}w_3yw_4\bar{y}$ is the coding of $G$, then $W_{\sigma_{x,y}(G)}=w_3xw_2\bar{x}w_1yw_4\bar{y}$ and $W_{G'}=w_3xw_4\bar{x}w_1yw_2\bar{y}$. Since $W_G$ and $W_{G'}$ are equal (up to cyclic permutation and relabelling), there exists $\phi\in\Mod(\Sigma_g)$ such that $\phi(G)=G'$. Let us compute $\phi$. 

\begin{figure}[htbp]
\begin{center}
\includegraphics[scale=0.15]{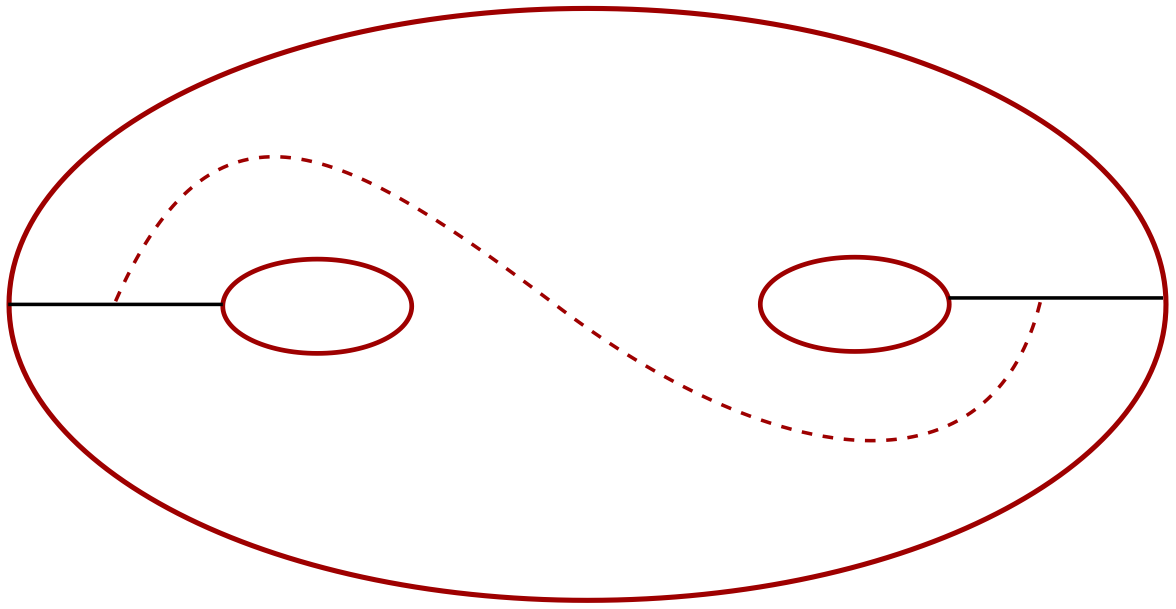}\hspace{0,5cm}
\includegraphics[scale=0.15]{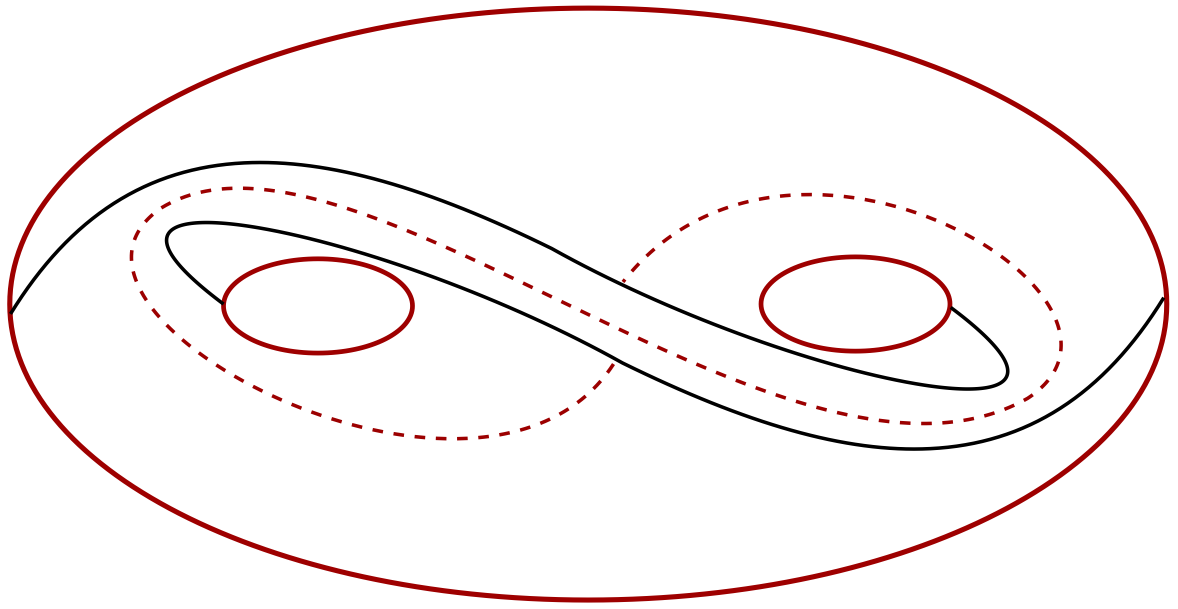}\hspace{0,5cm}
\includegraphics[scale=0.15]{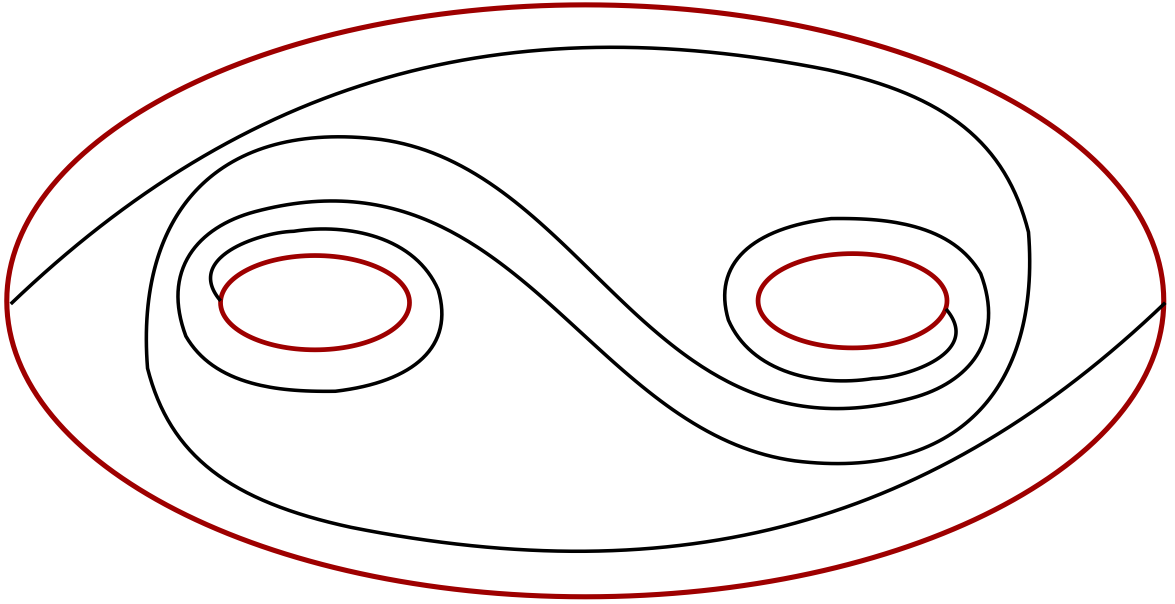}
\put(-204,19){\Large{$\rightarrow$}}
\put(-101,19){\Large{$\rightarrow$}}
\caption{Computing the effect of successive surgeries inside the pair of pants $\mathcal{P}$.}
\label{pantsuite}
\end{center}
\end{figure}

Since $x$ and $y$ are intertwined, $\lambda_{x,y}$ and $\lambda_{\bar{x},\bar{y}}$ intersect once and $\lambda:=\lambda_{x,y}\cup\lambda_{x,y}$ is a closed curve that self-intersects once. Taking a regular neighborhood of $\lambda$, we get a pair of pants $\mathcal{P}$ whose boundary $\beta_1$ (respectively $\beta_2$, $\beta_3$) intersects $G$ twice at $x$ and $y$ (respectively once each, at $x$ and $y$). So, $G$ and $G'$ differ only inside $\mathcal{P}$ and the sequence of surgeries depicted in \Cref{pantsuite} shows that $\phi=T_{\beta_2}^2T_{\beta}T_{\beta_3}^2$.
\end{proof}

\begin{definition}
Let $G$ be a unicellular graph on $\Sigma_g$. A Dehn twist $T_{\beta}$ (or the underlying curve $\beta$) is called \boldit{visible} relative to $G$ if $\beta$ intersects $G$ at two points contained in distinct intertwined edges. A Dehn twist $T_{\beta}$ is \boldit{admissible} if $T_{\beta}\in\Mod(\Sigma_g)[G]$.  
\end{definition} 

As a direct consequence of Lemma \ref{lemma:twists}, we have a very useful criterion.
\begin{corollary}\label{cor:visibletwist}
Let $G$ be a unicellular graph on $\Sigma_g$ and $\beta$ a simple curve that is visible relative to $G$. Then $T_{\beta} \in \Mod(\Sigma_g)[G]$. 
\end{corollary}

For the remainder of the article, we will restrict attention to the case of {\em unicellular collections}- recall that these are unicellular graphs with the degree partition $d_4 = (4, \dots, 4)$, and consist of a union of simple curves. Visible Dehn twists with respect to a unicellular collection are admissible, and we will provide a set of admissible Dehn twists that generate $\Mod(\Sigma_g)[\Gamma]$.

Let $\Sigma_{g,n}$ denote the oriented surface of genus $g$ with $n\le 2$ boundary components (although we have been considering closed surfaces $\Sigma_g$ thus far, in \Cref{lemma:nextxi}, we will need to consider subsurfaces of $\Sigma_g$, necessarily with boundary). We recall that a collection $\mathcal{C}:=\{\gamma_1,\dots,\gamma_n\}$ of simple curves on $\Sigma_{g,n}$ is a \boldit{chain} if $i(\gamma_i,\gamma_{i+1})=1$ and $i(\gamma_k,\gamma_l)=0$ otherwise (here and throughout, $i(\cdot, \cdot)$ denotes the geometric intersection number).  The maximal number of curves in a chain $\mathcal{C}$ on $\Sigma_{g}$ or $\Sigma_{g,1}$ is $2g+1$, and is $2g+2$ on $\Sigma_{g,2}$.

When the chain has $2g$ simple curves, it turns out that it is a unicellular collection. Since $\mathcal{K}_{d_4,g}$ is connected for all $g\geq 1$ (\Cref{theorem:previous}.1), every connected component of $\widetilde{K}_{d_4,g}$ contains a unicellular collection which is a chain; let us fix one particular such $\Gamma_0$  and let $\gamma_0$ be a simple curve obtained by smoothing vertices of $\Gamma_0$ (see \Cref{origine}), representing the mod-$2$ homology class $[\Gamma_0]$. Our final result in this section exhibits a set of Dehn twists in the stabilizer $\Mod(\Sigma_g)[\Gamma_0]$; in the next section, we will see that these twists suffice to generate the stabilizer of the mod-$2$ homology class $[\Gamma_0]$.

\begin{figure}[htbp]
\begin{center}
\includegraphics[scale=0.5]{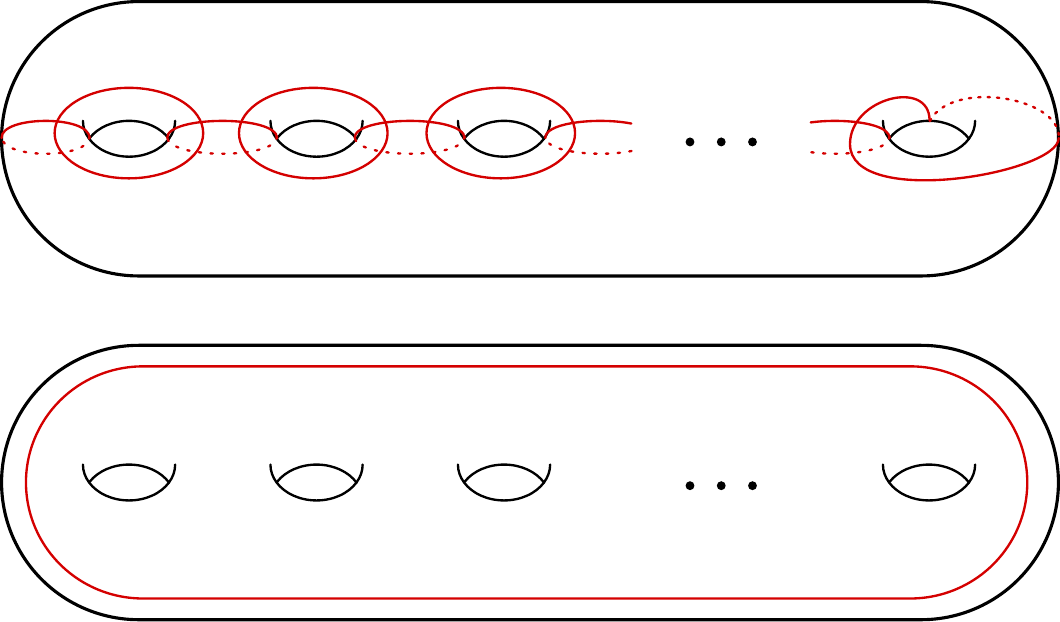}
\caption{The unicellular collection $\Gamma_0$ on top, and a simple curve representing its mod-$2$ homology class $\gamma_0$ on bottom. }
\label{origine}
\end{center}
\end{figure}

\begin{figure}[htbp]
\begin{center}
\labellist
\scriptsize
\pinlabel $\beta_1$ [t] at 62.36 92.87
\pinlabel $\beta_2$ [b] at 104.87 76.53
\pinlabel $\beta_3$ [t] at 150.22 92.87
\pinlabel $\gamma$ [bl] at 161.56 99.20
\pinlabel $\beta_{2g-2}$ [br] at 393.98 76.53
\pinlabel $\beta_{2g-1}$ [br] at 445.83 90.70
\pinlabel $\alpha$ at 515 75
\endlabellist
\includegraphics[scale=0.7]{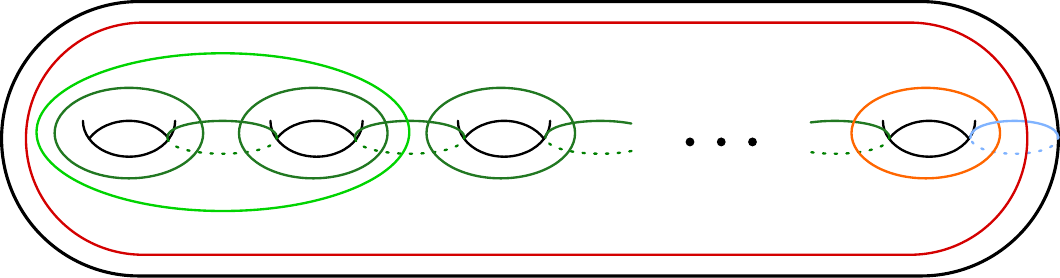}
\caption{The curves $\alpha, \beta_1, \dots, \beta_{2g-2}, \beta_{2g-1}, \gamma$ of \Cref{prop:genexhibit}.}
\label{curveconfig}
\end{center}
\end{figure}
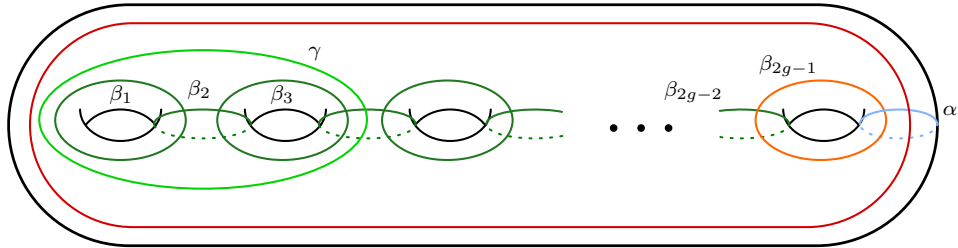

\begin{proposition}\label{prop:genexhibit}
    With reference to \Cref{curveconfig}, the following mapping classes are contained in $\Mod(\Sigma_g)[\Gamma_0]$:
    \[
     T_\alpha^2, T_{\beta_1}, \dots, T_{\beta_{2g-2}}, T_{\beta_{2g-1}}, T_\gamma.
    \]
\end{proposition}
\begin{proof}
Comparing \Cref{origine} and \Cref{curveconfig}, one sees that the curve $\alpha$ intersects $\Gamma_0$ at a single point. By \Cref{lemma:twists}, $T_\alpha^2 \in \Mod(\Sigma_g)[\Gamma_0]$ as claimed.

We next argue that $T_{\beta_1}, \dots, T_{\beta_{2g-2}}, T_\gamma \in \Mod(\Sigma_g)[\Gamma_0]$. As illustrated in \Cref{visible}, the unicellular collection $\Gamma_0$ is such that two unoriented edges with no common vertex are intertwined with respect to some choice of orientations. This implies that the Dehn twists along the green curves $\beta_1, \dots, \beta_{2g-2}, \gamma$ shown in \Cref{curveconfig} are visible with respect to $\Gamma_0$, and hence contained in $\Mod(\Sigma_g)[\Gamma_0]$ by \Cref{cor:visibletwist}. 
\begin{figure}[htbp]
\begin{center}
\includegraphics[scale=0.35]{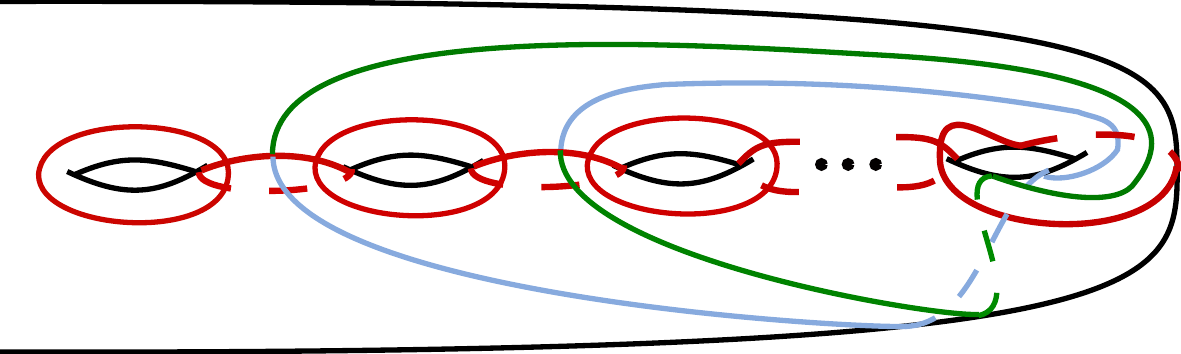}\hspace{0.6cm}\includegraphics[scale=0.35]{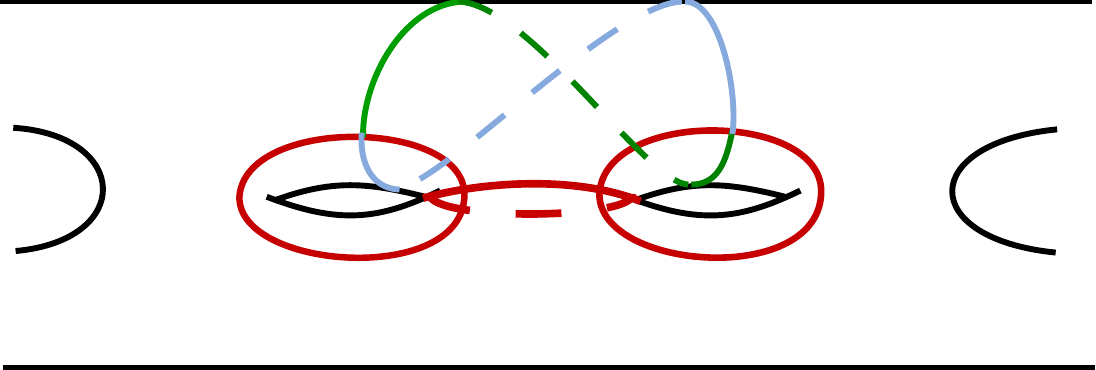}
\caption{Example of intertwined edges on $\Gamma_0$. Note how the arcs $\lambda_{x,y}$ and $\lambda_{\bar{x},\bar{y}}$ (shown in blue, resp. green) intersect once.}
\label{visible}
\end{center}
\end{figure}

It remains to show that $T_{\beta_{2g-1}} \in \Mod(\Sigma_g)[\Gamma_0]$. This will require more work, since $\beta_{2g-1}$ is not visible on $\Gamma_0$. We show below that $\beta_{2g-1}$ is visible on the unicellular collection $\Gamma_1$, and hence by \Cref{cor:visibletwist}, $T_{\beta_{2g-1}} \in \Mod(\Sigma_g)[\Gamma_1]$. \Cref{d1} shows that $\Gamma_1$ is obtained from $\Gamma_0$ by a surgery. It follows by \Cref{lemma:conjequal} that $\Mod(\Sigma_g)[\Gamma_1] = \Mod(\Sigma_g)[\Gamma_0]$ and hence $T_{\beta_{2g-1}} \in \Mod(\Sigma_g)[\Gamma_0]$ as required.

\begin{figure}[htbp]
\begin{center}
\includegraphics[scale=0.35]{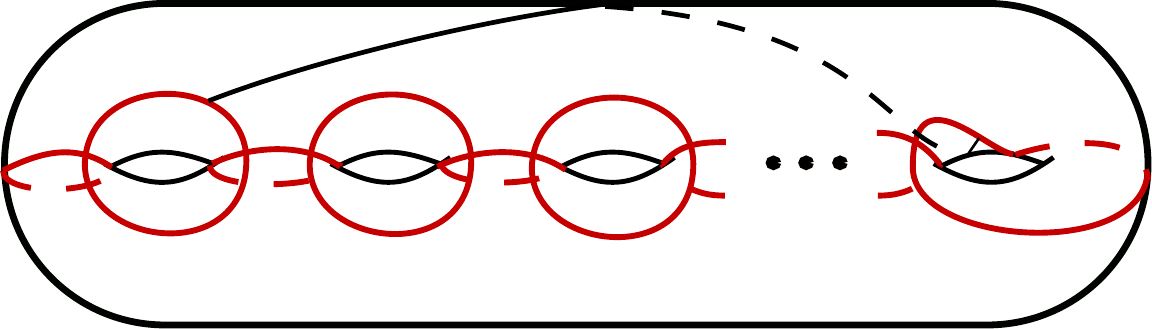}\hspace{2cm}\includegraphics[scale=0.35]{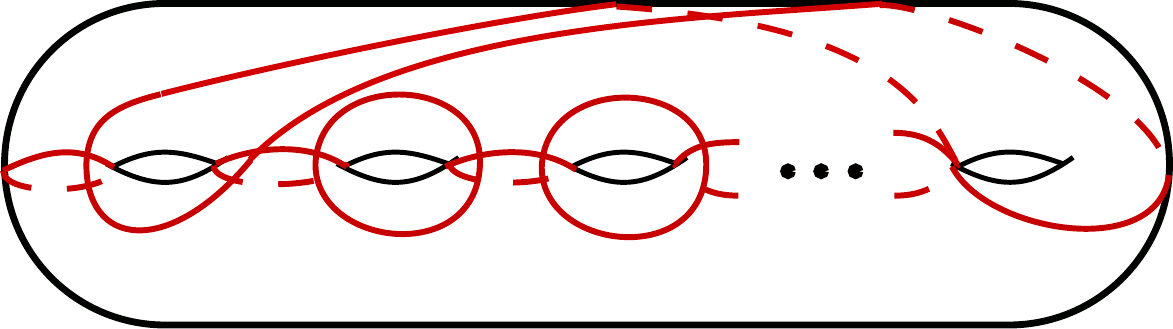}\put(-232,20){{$\rightarrow$}} \put(-7,0){$\Gamma_1$} \put(-262,0){$\Gamma_0$}
\caption{The surgery taking $\Gamma_0$ to $\Gamma_1$.}
\label{d1}
\end{center}
\end{figure}
In general, a curve $\alpha$ of a unicellular collection $\Gamma$ is called \boldit{1-simple} if $\alpha$  is simple and intersects $\Gamma\setminus\alpha$ once. If $(x,\bar{x})$ intersects a 1-simple curve $\alpha$ of $\Gamma$ then $(x,\bar{x})$ is intertwined with any edge of $\Gamma$ except $\alpha:=(a,\bar{a})$. In fact, if $x$ is the oriented edge toward $\alpha$ then the coding of $\Gamma$ is of the form $xa\bar{x}w$ and this implies that $x$ is intertwined with either $y$ or $\bar{y}$ for any other edge $y$ not equal to $a$.

In $\Gamma_1$ there is exactly one 1-simple curve $\alpha$ and the arc $\lambda$ along which the surgery on $\Gamma_0$ is made splits into two edges one of which is adjacent to $\alpha$; let us denote it $(x,\bar{x})$. It follows that $(x,\bar{x})$  is intertwined with any other edge in $\Gamma_1$; in particular with the edges $(t,\bar{t})$ and $(z,\bar{z})$ adjacent to $(x,\bar{x})$ that are not in $\alpha$. One of the visible Dehn twists supported by  $(t,\bar{t})$ and $(x,\bar{x})$ or by  $(z,\bar{z})$ and $(x,\bar{x})$ is along $\beta_{2g-1}$.      
\end{proof}

\section{Generating the stabilizer of a mod-$2$ homology class}\label{section:genset}

\begin{definition}\label{def:xistab}
    Let $\xi$ be an isotopy class of oriented simple closed curve on $\Sigma_g$. Define
    \[
    \Mod(\Sigma_g,\xi) \le \Mod(\Sigma_g)
    \]
    as the subgroup of elements preserving $\xi$. 

    Let $S_\xi \subset \Sigma_g$ be a subsurface of positive genus with two boundary components, both isotopic to $\xi$ as unoriented curves (such $S_\xi$ is unique up to isotopy). Recall \cite[Proposition 3.20]{farmar} that the inclusion $S_\xi \subset \Sigma_g$ induces a {\em surjection} 
\[
\Mod(S_\xi) \onto \Mod(\Sigma_g,\xi),
\]
and hence any set of generators for $\Mod(S_\xi)$ induces a set of generators for $\Mod(\Sigma_g,\xi)$.
\end{definition}

\begin{theorem}\label{theorem:genset}
    Let $g \ge 3$ be given, and let $x \in H_1(\Sigma_g; \Z/2\Z)$ be nonzero. Then $\Mod(\Sigma_g)[x]$ is generated by $T_\eta^2$ and $\Mod(\Sigma_g,\xi)$, where $\xi$ is any simple closed curve satisfying $[\xi] = x \in H_1(\Sigma_g; \Z/2\Z)$, and $\eta$ is any simple closed curve with geometric intersection $i(\xi, \eta) = 1$.
\end{theorem}

\begin{proof}
    Define
    \[
    \mathcal{G} := \pair{T_\eta^2, \Mod(\Sigma_g,\xi)}.
    \]
    There is an evident containment $\mathcal{G} \le \Mod(\Sigma_g)[x]$. We will argue that this in an equality in two steps. Let $f \in \Mod(\Sigma_g)[x]$ be given. In Step 1, we will produce $f_1$ in the coset $\mathcal{G} f$ that preserves the {\em integral} homology class of $\xi$. Then in Step 2, we will produce $f_2 \in \mathcal{G} f$ that preserves the oriented isotopy class of $\xi$, so that $f_2 \in \Mod(\Sigma_g,\xi) \le \mathcal{G}$, ultimately showing that $f \in \mathcal{G}$. To avoid unnecessary notational complexity, we will not track the specific modifications made to $f$ by $\mathcal{G}$, instead speaking of {\em adjusting} $f$ (tacitly always by left multiplication by some $\gamma \in \mathcal{G}$).

    \para{Step 1: Preserving integral homology} Recall that a {\em symplectic basis} for $H_1(\Sigma_g; \Z)$ is a generating set $x_1, y_1, \dots, x_g, y_g$ such that $\pair{x_i,y_i} = 1$ and all other pairings are zero. We recall the well-known fact that a pair of elements $x,y \in H_1(\Sigma_g ;\Z)$ satisfying $\pair{x,y} = 1$ can be extended to a symplectic basis with $x_1 =x$ and $y_1 = y$, and apply this to $x = [\xi]$ and $y = [\eta]$. 
    \begin{remark} \label{remark:classes}
    In these coordinates, classes in the subspace 
    \[
    H_1(\Sigma_g \setminus \xi; \Z) := \pair{x_1, x_2, y_2, \dots, x_g, y_g}
    \]
    admit representatives as simple closed curves supported on $S_\xi$. 
    \end{remark}
    We also define the symplectic subspace
    \[
    H' = \pair{x_2, y_2, \dots, x_g, y_g}.
    \]

    By hypothesis, $f(x_1) = x_1 \pmod 2$, so that
    \[
    f(x_1) = a x_1 + b x_2 + w,
    \]
    where $a$ is odd, $b$ is even, and $w = 2 w'$ for some $w' \in H'$. Define
    \[
    c:= \gcd(w) \qquad \mbox{and} \qquad d:= \gcd(a,b);
    \]
    note that $c$ is even, $d$ is odd, and $\gcd(c,d) = 1$. 

    \para{Claim 1} {\em $f$ can be adjusted so that $f(x_1) = d x_1 + 2d y_1 + c x_2$.} 
    \proof Note first that $\Mod(\Sigma_g,\xi)$, and hence $\mathcal{G}$, acts transitively on elements of $H'$ of given gcd, while leaving $x_1$ and $y_1$ components unchanged. Thus $f$ can be adjusted so that $f(x_1) = a x_1 + b y_1 + c x_2$. 
    
    The second adjustment can be performed via a variant of the Euclidean algorithm. The Dehn twist $T_\xi^{\pm 1} \in \Mod(\Sigma_g,\xi) \le \mathcal{G}$ takes the vector $ax_1 + by_1 + c x_2$ to $(a\pm b)x_1 + b y_1 + c x_2$, and $T_\eta^{\pm 2} \in \mathcal{G}$ takes $ax_1 + b y_1 + cx_2$ to $ax_1 + (b\mp 2a) y_1 + c x_2$. If $0 < \abs{b} < \abs{a}$, then, for appropriate choice of sign, $\abs{a\pm b} < \abs{a}$. If instead $\abs{b} > \abs{a} > 0$, then likewise $\abs{b \mp 2a} < \abs{b}$ for appropriate choice of sign. In this way, $ax_1 + by_1 + cx_2$ can be taken to $dx_1 + c x_2$, and hence to $dx_1 + 2d y_1 + c x_2$, via a final application of $T_\eta^2$. \qed

    \para{Claim 2} {\em $f$ can further be adjusted so that $f(x_1) = (c-d) x_1 + 2d y_1 + 2d' x_2$ for some integer $d'$.}
    \proof By \Cref{remark:classes}, there is a simple closed curve $\zeta \subset S_\xi$ with $[\zeta] = x_1 + y_2$ in $H_1(\Sigma_g; \Z)$. The twist $T_\zeta \in \mathcal{G}$ then takes $dx_1 + 2d x_2 + c x_2$ to $(c-d)x_1 + 2d y_1 + (2c-2d) x_2$ as claimed. \qed

    \para{Claim 3} {\em $f$ can further be adjusted so that $f(x_1) = x_1 + 2y_1 + 2d'x_2$.} 
    \proof Following Claim 2, this now follows by repeating the argument of Claim 1, noting that since $\gcd(c,d) = 1$ and $c$ is even and $d$ is odd, also $\gcd(c-d, 2d) = 1$. \qed
    
    To finish the proof of Step 1, we observe that by \Cref{remark:classes}, there is a simple closed curve $\zeta' \subset S_\xi$ with $[\zeta'] = x_1 + x_2$. Acting on $H_1(\Sigma_g; \Z)$,
    \[
    T_\eta^2 T_\xi^{-d'} T_{\zeta'}^{d'}(x_1 + 2y_1 + 2d' x_2) = T_\eta^2 (x_1 + 2y_1) = x_1.
    \]

    \para{Step 2: Preserving isotopy} Following Step 1, given $f \in \Mod(\Sigma_g)[x]$, we can adjust $f$ within $\mathcal{G} f$ to assume that $f(x) = x$ in $H_1(\Sigma_g; \Z)$. We will now show how such $f$ can be adjusted so that $f(\xi) = \xi$ on the level of isotopy classes of oriented curves.

    \begin{lemma} \label{lemma:allsquares}
        Let $\gamma$ be an isotopy class of simple closed curve on $\Sigma_g$ that satisfies $i(\xi, \gamma) = 1$. Then $T_\gamma^2 \in \mathcal{G}$.
    \end{lemma}
    \begin{proof}
    Choose representatives for $\gamma$ and $S_\xi$ in minimal position; by abuse of notation we continue to refer to these by the same symbols. Since $i(\gamma, \xi) = 1$ and the representatives are in minimal position, it follows that $\gamma \cap S_\xi$ is a single arc connecting the boundary components of $S_\xi$. Possibly after several applications of $T_\xi$, also $\gamma \cap (\Sigma_g \setminus S_\xi)$ is isotopic to $\eta \cap (\Sigma_g \setminus S_\xi)$. As $\Mod(S_\xi)$ acts transitively on isotopy classes of arcs connecting the boundary components, it follows that there is $g \in \Mod(S_\xi)$ taking $\gamma \cap S_\xi$ to $\eta \cap S_\xi$. Thus $T_\xi^k g (\gamma) = \eta$. As $T_\eta^2 \in \mathcal{G}$, it follows that also $T_\gamma^2 \in \mathcal{G}$ as claimed. 
    \end{proof}

    \begin{lemma}\label{lemma:nextxi}
        Let $\xi'$ be an oriented simple closed curve such that $[\xi] = [\xi']$ in $H_1(\Sigma_g;\Z)$ and such that $i(\xi, \xi') = 0$. Then there is an element $\tau \in \mathcal{G}$ such that $\tau(\xi) = \xi'$. Consequently, $\mathcal{G}$ contains the element $T_{\xi'}$ and the subgroup $\Mod(\Sigma_g)[\xi']$.
    \end{lemma}
\begin{figure}[ht]
\labellist
\small
\pinlabel $\beta$ [l] at 247 102.04
\pinlabel $\beta'$ [l] at 247 25.51
\pinlabel $\xi$ [tl] at 255 75
\pinlabel $\xi'$ [bl] at 277.77 107.71
\pinlabel $\alpha_2$ [b] at 141.72 93.53
\pinlabel $\alpha_{2k+1}$ [br] at 3 76.53
\pinlabel $S$ at 60 30
\endlabellist
\includegraphics[scale=0.8]{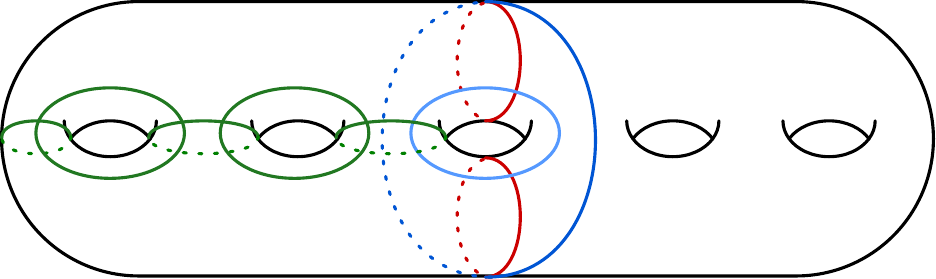}
\caption{The configuration of curves involved in \Cref{lemma:nextxi}, illustrated for $g = 5$ and $k = 2$. The label on $\alpha_1$ is not shown for clarity.}
\label{figure:bp}
\end{figure}
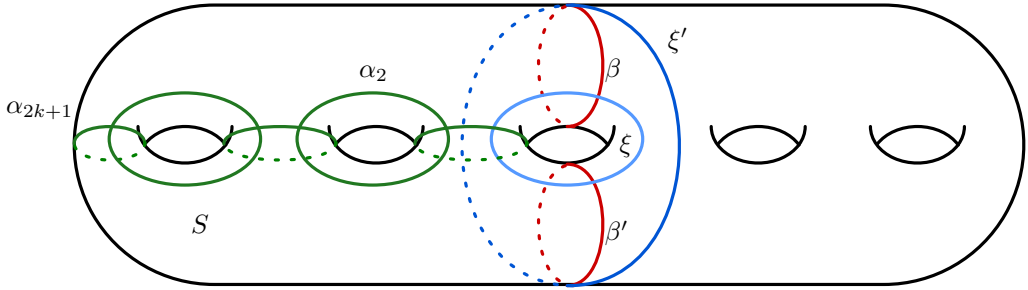
    
    \begin{proof}
        By the change-of-coordinates principle, $\xi$ and $\xi'$ can be depicted in the form shown in \Cref{figure:bp}. As shown there, there are curves $\beta, \beta'$ such that $T_\beta T_\beta'^{-1}(\xi) = \xi'$. We will show that $T_\beta T_{\beta'}^{-1} \in \mathcal{G}$.

        To see this, let $S \subset \Sigma_g$ be either of the subsurfaces bounded by $\beta \cup \beta'$. Again by the change-of-coordinates principle, there is a maximal chain $\alpha_1, \dots, \alpha_{2k+1} \subset S$ such that $i(\alpha_1, \xi) = 1$ and such that $i(\alpha_j, \xi) = 0$ for $j \ge 2$. Necessarily then $\alpha_j$ is supported on $\Sigma_g \setminus \xi$ for $j \ge 2$, so that $T_{\alpha_j} \in \mathcal{G}$ for $j  \ge 2$. By \Cref{lemma:allsquares}, since $i(\alpha_1, \xi) = 1$, also $T_{\alpha_1}^2 \in \mathcal{G}$.

        By the alternate formulation of the chain relation \cite[Section 4.4.1]{farmar}, $T_\beta T_{\beta'}$ is an element of the subgroup generated by $T_{\alpha_1}^2$ and $T_{\alpha_j}$ for $j \ge 2$, and {\em a fortiori} is an element of $\mathcal{G}$. By \Cref{lemma:allsquares}, also $T_\beta^2 \in \mathcal{G}$, so that $T_\beta T_{\beta'}^{-1} \in \mathcal{G}$.

        The final claim follows by noting that conjugation by $T_\beta T_{\beta'}^{-1}$ takes $T_\xi$ to $T_{\xi'}$ and $\Mod(\Sigma_g,\xi)$ to $\Mod(\Sigma_g)[\xi']$.
    \end{proof}

    We can now complete the proof of Step 2. According to \cite[Theorem 1.9]{putmantrick}, since $g \ge 3$, there is a sequence of oriented simple closed curves $\xi_1, \dots, \xi_n \subset \Sigma_g$ with $\xi_1 = \xi$ and $\xi_n = f(\xi)$, such that $[\xi_i] = [\xi]$ for all $i$,  and such that $i(\xi_j, \xi_{j+1}) = 0$ for all $j < n$. By repeated application of \Cref{lemma:nextxi}, $\mathcal{G}$ contains elements $\tau_1, \dots, \tau_{n-1}$ such that $\tau_{j}(\xi_{j+1}) = \xi_j$ for all $1 \le j \le n-1$. Then
    \[
    \tau_1 \dots \tau_{n-1} f(\xi) = \xi,
    \]
    and $\tau_1 \dots \tau_{n-1} f \in \mathcal{G}$, completing Step 2 and the argument as a whole.
    \end{proof}

We can now complete the proof of \Cref{maintheorem}.

\begin{figure}[ht]
\includegraphics[scale=0.6]{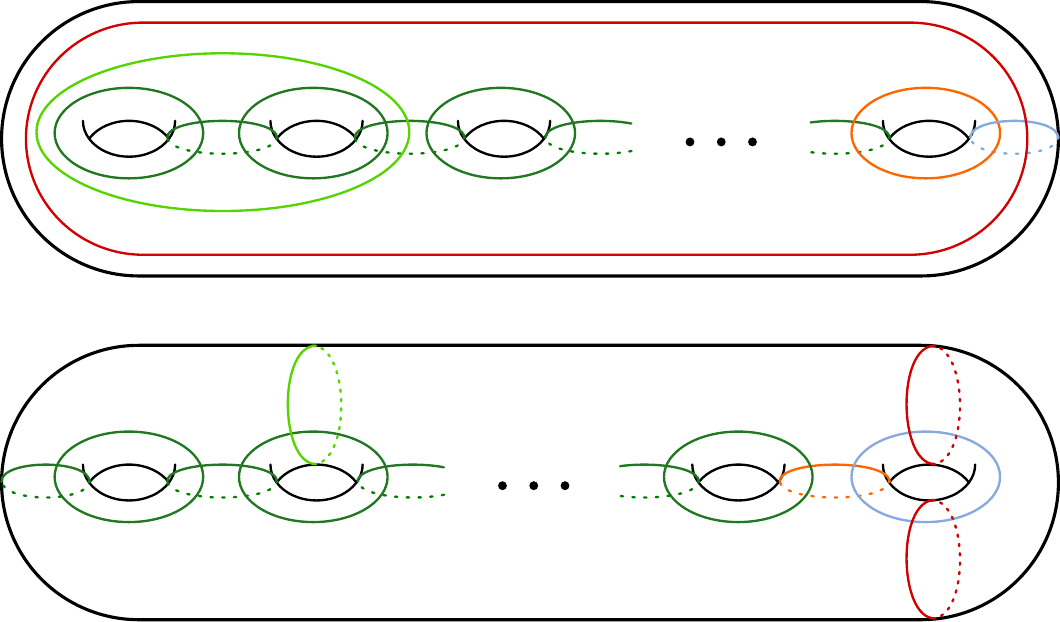}
\caption{Identifying the configuration in \Cref{curveconfig} with the Humphries generating set for $\Mod(\Sigma_{g-1,2})$.}
\label{secrethumphries}
\end{figure}

\begin{proof}[Proof of \Cref{maintheorem}]
    Let $\Gamma \in \tilde{\mathcal{K}}_{d_4,g}$ be given. By \Cref{lemma:conjequal}, if $\Gamma'$ is in the same component of $\tilde{\mathcal{K}}_{d_4,g}$, then there is an equality of stabilizer groups $\Mod(\Sigma_g)[\Gamma] = \Mod(\Sigma_g)[\Gamma']$, and $\Mod(\Sigma_g)[\Gamma]$ is conjugate to {\em any} stabilizer group $\Mod(\Sigma_g)[\Gamma'']$ regardless of whether $\Gamma$ and $\Gamma''$ are in the same component. It follows that without loss of generality, we can take $\Gamma = \Gamma_0$ as in \Cref{origine}. 
    
    Let $x \in H_1(\Sigma_g; \Z/2\Z)$ be the surgery invariant of $\Gamma_0$. Resolving the crossings in $\Gamma_0$ as shown in \Cref{origine}, one obtains the simple closed curve $\xi$ shown therein as an explicit representative for $x$.
    
    By \Cref{corollary:stabsubgroup}, there is a containment $\Mod(\Sigma_g)[\Gamma_0] \le \Mod(\Sigma_g)[x]$. To show the opposite containment, we appeal to \Cref{theorem:genset}, with $\xi$ as in \Cref{origine} and $\eta = \alpha$ as in \Cref{curveconfig}. By \Cref{prop:genexhibit}, the mapping classes $T_\alpha^2, T_{\beta_1}, \dots, T_{\beta_{2g-1}}$, and $T_\gamma$ all belong to $\Mod(\Sigma_g)[\Gamma_0]$. \Cref{secrethumphries} shows that the curves $\beta_1, \dots, \beta_{2g-1}, \gamma$ form the configuration of the Humphries generating set on the subsurface $\Sigma_\xi$. Thus by \Cref{theorem:genset}, it follows that $T_{\beta_1}, \dots, T_{\beta_{2g-1}}, T_\gamma, T_\alpha^2$ together generate $\Mod(\Sigma_g)[x]$, showing the desired equality $\Mod(\Sigma_g)[\Gamma_0] = \Mod(\Sigma_g)[x]$.

    It remains to prove the final claim, that if $\Gamma$ and $\Gamma'$ have the same surgery invariant $[\gamma] = [\gamma'] = x$, then $\Gamma$ and $\Gamma'$ are contained in the same component of $\widetilde{\mathcal K}_{d_4,g}$. By \Cref{theorem:previous}.1, the combinatorial surgery graph $\mathcal K_{d_4,g} = \tilde{\mathcal K}_{d_4,g}/ \Mod(\Sigma_g)$ is connected, so that $\Gamma'$ is connected via a sequence of surgeries to $\Gamma''$ of the form $\Gamma'' = f(\Gamma)$ for some $f \in \Mod(\Sigma_g)$. As the surgery invariants of $\Gamma$ and $\Gamma'$ (and hence $\Gamma''$) all equal $x \in H_1(\Sigma_g; \Z/2\Z)$ by hypothesis, it follows that $f  \in \Mod(\Sigma_g)[x]$. By the above, it follows that $f \in \Mod(\Sigma_g)[\Gamma]$, showing that there is a sequence of surgeries connecting $\Gamma$ to $\Gamma''$ and ultimately to $\Gamma'$.
\end{proof}

    \bibliography{references}{}
	\bibliographystyle{alpha}

\end{document}